\documentclass{article}

\usepackage{amsmath}
\usepackage{amssymb}
\usepackage{amsthm}

\usepackage{graphicx}
\usepackage{float}
\usepackage{bm}
\usepackage{enumitem}
\usepackage{mathdots}
\usepackage{booktabs}
\usepackage{rotating}
\usepackage{listings}
\usepackage{hyperref}
\usepackage{xcolor}
\definecolor{Mycolor}{RGB}{10,74,38}
\hypersetup{
  colorlinks=true,
  linkcolor=Mycolor,
  citecolor=Mycolor,
  urlcolor=Mycolor
}

\usepackage[ruled,linesnumbered,longend]{algorithm2e}
\usepackage[a4paper,left=2.8cm,right=2.8cm,top=2.5cm,bottom=2.5cm]{geometry}
\usepackage{fancyhdr}
\usepackage[framemethod=tikz]{mdframed}
\newtheorem{theorem}{Theorem}[section]
\newtheorem{corollary}{Corollary}[section]
\newtheorem{proposition}{Proposition}[section]
\newtheorem{assumption}{Assumption}[section]

\newtheorem{lemma}{Lemma}[section]

\makeatletter 
\@addtoreset{equation}{section}
\makeatother  

\newtheorem{remark}{Remark}[section]

\usepackage{url}
\usepackage{mathtools}
\usepackage{lipsum}
\usepackage{tcolorbox}

\newcommand{\Sph}{\mathbb{S}^2}
\newcommand{\Nzero}{\mathbb{N}_0}
\newcommand{\Pn}{\mathcal P_n}
\newcommand{\Ln}{\mathcal L_n}
\newcommand{\Qn}{\mathcal Q_N}
\newcommand{\Qip}[2]{\langle #1,#2\rangle_{\mathcal Q}}
\newcommand{\ip}[2]{\langle #1,#2\rangle}
\newcommand{\abs}[1]{\left\lvert #1\right\rvert}
\newcommand{\dd}{\,\mathrm d}
\newcommand{\Id}{I}
\newcommand{\Top}{\mathcal T}
\newcommand{\Mop}{\mathcal M}
\newcommand{\Rop}{\mathcal R}
\newcommand{\Bop}{\mathcal B}

\begin{document}

\title{Superconvergence and aliasing saturation in Sloan iteration for spherical integral equations}

\author{ Hao-Ning Wu\footnotemark[2]
       }

\renewcommand{\thefootnote}{\fnsymbol{footnote}}
\footnotetext[2]{School of Mathematical Sciences, Xiamen University, Xiamen, Fujian, 361005, China (hnwu@connect.hku.hk)}

\maketitle

\begin{abstract}
Sloan iteration raises the convergence order of Galerkin and
degenerate-kernel approximations to second-kind integral equations.
After quadrature discretization, these become a discrete Galerkin method and
a product-integration Nystr\"om method, respectively.  How much of this
improvement survives quadrature discretization?  For zonal integral equations
on the sphere, we give a sharp answer by expressing the Sloan error identity
in terms of  spherical harmonics.
In the absence of quadrature, Sloan iteration fully exploits the smoothing of the
integral operator.  Quadrature can destroy this gain by aliasing unresolved
information into low-frequency modes, where further iteration no longer
improves the asymptotic rate.  This yields a unified tail--aliasing
description of these four methods.
For positive-weight polynomially exact quadrature and multipliers of exact
algebraic order, we derive sharp two-sided worst-case bounds that separate the
spectral-tail and aliasing contributions.  Their balance yields a
depth-dependent recovery--saturation threshold: sufficient overintegration
recovers the quadrature-free rate, while below the threshold aliasing
determines the sharp order.  We also extend the analysis beyond polynomial
exactness using Marcinkiewicz--Zygmund stability and Gram-corrected least
squares.  Numerical experiments illustrate both regimes.
\end{abstract}



\section{Introduction}

Degenerate-kernel, projection, and Nystr\"om methods are standard
discretizations of linear integral equations of the second kind; see, e.g.,
\cite{Anselone1971,Atkinson1997Numerical,MR1723850}.  A classical feature of
these methods is that the accuracy of an approximate solution may improve
substantially after one further application of the integral equation.  This
is the basis of Sloan iteration, introduced in
\cite{Sloan1976Degenerate,Sloan1976Improvement}.  For compact integral operators with suitable smoothing properties, 
repeated application of the operator to the approximation error can yield higher-order convergence.
In computation, however, the starting approximation is typically constructed
using quadrature, and the resulting integration error may prevent the
iterates from attaining their quadrature-free convergence rates.  We ask how
accurate the quadrature must be to recover these rates, and what convergence
remains when this accuracy is insufficient.  For zonal operators on the
sphere, we answer these questions sharply after any fixed number of Sloan
steps, provided the spherical harmonic multipliers have exact algebraic order.

\subsection{Sloan iteration and multiplier setting}

Let \(X\) be a Banach space, \(f\in X\), and
\(\Top:X\to X\) be compact with \(\Id-\Top\) invertible.  We consider
\begin{equation}\label{eq:model}
  (\Id-\Top)u=f.
\end{equation}
For an approximation \(v\) to \(u\), we define the Sloan iterates by $\mathcal S_0(v)=v$ and
$\mathcal S_{q+1}(v)=f+\Top\mathcal S_q(v)$ for $q\in\Nzero$.
Subtracting the recursion from the fixed-point identity
\(u=f+\Top u\) and iterating gives
\begin{equation}\label{eq:error-intro}
  u-\mathcal S_q(v)=\Top^q(u-v).
\end{equation}
Thus each Sloan step applies one additional power of the integral operator to
the error.  This identity underlies the classical superconvergence theory for
projection and degenerate-kernel methods; see, e.g.,
\cite{ChenXu1998,GrahamJoeSloan1985,KanekoPadillaXu2001,KanekoXu1996,
Sloan1984FourVariants,SloanThomee1985}.  Repeated and multi-projection variants
have also been studied in, e.g.,
\cite{ChakrabortyAgrawalNelakanti2025,LongNelakanti2007,
MandalNelakanti2019,PorterStirling1993}.

We specialize to \(X=L^2(\Sph)\), where \(\Sph\subset\mathbb R^3\) is the unit
sphere equipped with surface measure \(\omega(\Sph)=4\pi\), and consider the
zonal integral operator
\begin{equation}\label{eq:zonal-intro}
  (\Top v)(x)=\int_{\Sph}\Phi(x\cdot y)v(y)\dd\omega(y),
  \qquad \Phi\in L^1(-1,1).
\end{equation}
Such operators form a standard class in spherical integral equations; see,
e.g., \cite{ColtonKress2013}.  If \(\{Y_{\ell,k}\}\) is a real orthonormal
spherical harmonic basis, the Funk--Hecke formula yields
\begin{equation}\label{eq:funk-hecke-intro}
  \Top Y_{\ell,k}=\mu_\ell Y_{\ell,k},
  \qquad
  \mu_\ell=2\pi\int_{-1}^{1}\Phi(t)P_\ell(t)\dd t.
\end{equation}
Thus \(\Top\) is diagonal in spherical harmonics, and Sloan iteration acts
independently on each harmonic degree through powers of the multiplier
\(\mu_\ell\).  Background on spherical harmonic analysis and approximation theory
can be found in \cite{MR2934227,DaiXu2013}. Throughout the analysis, we impose the following algebraic smoothing and
nonresonance conditions on the multiplier sequence \((\mu_\ell)_{\ell\ge0}\).
\begin{assumption}\label{ass:operator}
There exist constants \(\beta>0\), \(C_{\Top}>0\), and \(\delta>0\) such that for $\ell\in\Nzero$,
\begin{equation}\label{eq:operator-assumption}
  \abs{\mu_\ell}\leq C_{\Top}(1+\ell)^{-\beta},
  \qquad
  \abs{1-\mu_\ell}\geq\delta.\end{equation}
\end{assumption}
The first condition in \eqref{eq:operator-assumption} gives at least \(\beta\) orders of spectral smoothing,
whereas the second guarantees bounded invertibility of \(\Id-\Top\).
Assumption~\ref{ass:operator} suffices for the upper estimates.  To obtain
sharp rates, we also require a matching lower bound on the multipliers.

\begin{assumption}[Exact multiplier order]\label{ass:exact-order}
There exist constants \(c_{\Top}>0\) and \(\ell_0\in\Nzero\) such that
\begin{equation}\label{eq:exact-multiplier-order}
  \abs{\mu_\ell}\geq c_{\Top}(1+\ell)^{-\beta},
  \qquad \ell\geq\ell_0.
\end{equation}
\end{assumption}
Under Assumptions~\ref{ass:operator} and~\ref{ass:exact-order},
\(
  \abs{\mu_\ell}\asymp(1+\ell)^{-\beta}\) as
  $\ell\to\infty$,
so that \(\beta\) is the exact algebraic smoothing order of \(\Top\).
For example, the Laplace single-layer operator
\[
(Sv)(x)=\frac1{4\pi}\int_{\Sph}\frac{v(y)}{|x-y|}\,\dd\omega(y)
\]
satisfies \(SY_{\ell,k}=(2\ell+1)^{-1}Y_{\ell,k}\).  Thus, for
\(\Top=\kappa S\), \(0<\kappa<1\), one has
\(\mu_\ell=\kappa/(2\ell+1)\), and both assumptions hold with \(\beta=1\).

\subsection{Tail, aliasing, and the threshold}
Let \(\mathbb P_n\) denote the space of spherical harmonics of degree at most
\(n\), and let \(\Pn\) be the \(L^2\)-orthogonal projection onto
\(\mathbb P_n\).  With exact projection coefficients, the initial error is
supported only in degrees \(\ell>n\).  Equations~\eqref{eq:error-intro}
and~\eqref{eq:funk-hecke-intro} then show that each Sloan step damps this
spectral tail by an additional factor \(\mu_\ell\); see details in Section \ref{sec:setting}.

Quadrature changes the location of the error.  For a positive rule
$\Qn(g)=\sum_{j=1}^Nw_jg(x_j)$, replacing the exact Fourier--Laplace coefficients of
$\Pn$ gives the hyperinterpolation operator 
\begin{equation}\label{eq:hyperinterpolation}
  \Ln v
  =\sum_{\ell=0}^{n}\sum_{k=1}^{2\ell+1}
    \Qn(vY_{\ell,k})Y_{\ell,k},
\end{equation}
also proposed by Sloan in \cite{sloan1995polynomial}.
If $\Qn$ is exact on $\mathbb P_{2n}$, then $\Ln$ reproduces
$\mathbb P_n$, but generally $\Ln f\ne\Pn f$.  The identity
\begin{equation}\label{eq:split-intro}
  f-\Ln f=(\Id-\Pn)f+(\Pn f-\Ln f)
\end{equation}
separates the high-degree spectral tail from the low-degree aliasing term
$\Pn f-\Ln f\in\mathbb P_n$.  Quadrature therefore returns unresolved
information to the resolved space.  Further Sloan steps damp the spectral tail at high
degrees, but a fixed low-degree aliasing component need not gain any power of
$n$.  This is the mechanism behind aliasing saturation.  Related quadrature
effects in iterated methods were studied in, e.g.,
\cite{AtkinsonBogomolny1987,GolbergBowman1998,GolbergBowman1998Residual}.

For \(f\in H^r(\Sph)\), \(r>1\), we assume that the quadrature rule is exact
through degree \(t\ge 2n\), and set \(s=t-n\).  For the discrete Galerkin and
product-integration Nystr\"om methods, define $\gamma_G=q\beta$ and $\gamma_N=(q+1)\beta$,
respectively.  When \(\gamma_\star>1\), we prove a unified estimate
\begin{equation*}
  \|u-\mathcal S_q(u_{n,\mathcal Q}^{\star})\|_{L^\infty}
  \leq C\left(
    n^{-(r+\gamma_\star-1)}+s^{-r}
  \right)\|f\|_{H^r},
\end{equation*}
where \(\star\in\{G,N\}\) denotes the discrete Galerkin or
product-integration Nystr\"om method.  The first term is the filtered
spectral-tail contribution, while the second is the quadrature-aliasing
contribution.  Their balance defines the critical scale
$s_{n,\star}^{\rm crit}
  =n^{(r+\gamma_\star-1)/r}$.
  Above this scale, sufficient overintegration recovers the quadrature-free
rate of spectral tail.  Below it, a constant-mode argument combined with the optimal Sobolev
quadrature lower bound obtained in
\cite{BrauchartHesse2007,HesseSloan2005Lower}
yields the matching worst-case order \(s^{-r}\) for positive quadrature rules
of quasi-optimal cardinality.  Under
Assumption~\ref{ass:exact-order}, the spectral-tail rate is itself sharp, so
the critical scale becomes an intrinsic recovery--saturation
threshold at every fixed Sloan depth; see details in Section \ref{sec:discrete}.

We further replace polynomial exactness by Marcinkiewicz--Zygmund stability
\cite{filbir2011marcinkiewicz,LuWang2023,MR4124840,mhaskar2001spherical,
Wu2026SpectralMultiplier}.
Using Gram-corrected least squares, we retain the same tail--aliasing
decomposition, with explicit dependence on the Marcinkiewicz--Zygmund
stability constant; see Section~\ref{sec:inexact}.

\subsection{Contributions and organization}
Let \(\Rop:=(\Id-\Top)^{-1}\) and \(\Bop:=\Top\Rop\).
Table~\ref{tab:four-methods} summarizes the four starting approximations
considered in this paper.  Their representation formulas are
derived in the corresponding sections below.

\begin{table}[htbp]
  \centering
  \caption{The four basic starting methods.}
  \label{tab:four-methods}
  \begin{tabular}{lcc}
  \toprule
    starting method
      & orthogonal projection $\Pn$
      & hyperinterpolation $\Ln$ \\
    \midrule
    Galerkin
      & $u_n^G=\Rop\Pn f$
      & $u_{n,\mathcal Q}^G=\Rop\Ln f$ \\
    \shortstack[l]{degenerate kernel/Nystr\"om}
      & $u_n^D=f+\Bop\Pn f$
      & $u_{n,\mathcal Q}^N=f+\Bop\Ln f$ \\
    \bottomrule
  \end{tabular}
\end{table}

Our main contributions are as follows.
\begin{enumerate}[label=\textup{(\roman*)}]
\item We derive unified spherical harmonic error representations for the four
methods in Table~\ref{tab:four-methods}, separating the spectral tail from the
low-degree aliasing introduced by quadrature.
\item For positive polynomially exact quadrature, we establish matching upper
and lower bounds and identify the recovery--saturation scale
$s_{n,\star}^{\rm crit}$.  Under the exact-order multiplier assumption,
this scale is intrinsic and sharp.
\item For inexact sampling, we replace hyperinterpolation by Gram-corrected
least squares under a Marcinkiewicz--Zygmund inequality and derive
corresponding error estimates, which are further illustrated numerically.
\end{enumerate}
Here \emph{exact} means quadrature exactness of degree at least \(2n\), whereas
\emph{inexact} means that no such requirement is imposed.  In both cases the
starting approximation is numerical, while the Sloan map is defined using the
continuous operator \(\Top\); fully discrete iterations based on powers of a
numerical operator are not considered.  Throughout, the Sloan depth \(q\) is
fixed as \(n\to\infty\).

Section~\ref{sec:preliminaries} collects the spherical harmonic notation,
resolvent identities, and multiplier estimates.  Section~\ref{sec:setting}
treats exact projection coefficients, while Sections~\ref{sec:discrete} and
\ref{sec:inexact} analyze exact and inexact quadrature.  Numerical results are
reported in Section~\ref{sec:numerics}, followed by conclusions.

\section{Spherical harmonic setting and multiplier estimates}
\label{sec:preliminaries}

We work with real-valued functions on \(\Sph\).  The space \(C(\Sph)\) of continuous functions is
equipped with the uniform norm
$\|v\|_{L^\infty(\Sph)}
  =\max_{x\in\Sph}\abs{v(x)}$,
while \(L^2(\Sph)\) denotes the Hilbert space with inner product
$\langle v,z\rangle_{L^2(\Sph)}
  =\int_{\Sph}v(x)z(x)\dd\omega(x)$ and induced norm
$\|v\|_{L^2(\Sph)}
  =\langle v,v\rangle_{L^2(\Sph)}^{1/2}$.
Quadrature is understood pointwise on \(C(\Sph)\), whereas
orthogonal expansions and projections are taken in \(L^2(\Sph)\).

For \(\ell\in\Nzero\), a spherical harmonic of degree \(\ell\) is the
restriction to \(\Sph\) of a homogeneous harmonic polynomial of degree
\(\ell\) in \(\mathbb R^3\).  The degree-\(\ell\) spherical harmonics form a
\((2\ell+1)\)-dimensional subspace of \(L^2(\Sph)\) and satisfy
$-\Delta_{\Sph}Y_{\ell,k}
  =\ell(\ell+1)Y_{\ell,k}$,
where \(\Delta_{\Sph}\) denotes the Laplace--Beltrami operator on \(\Sph\).
For each \(\ell\), we choose an orthonormal basis
\(\{Y_{\ell,k}:1\leq k\leq2\ell+1\}\).  The resulting collection over all
degrees forms an orthonormal basis of \(L^2(\Sph)\).  The addition formula reads
\begin{equation}\label{eq:addition}
  \sum_{k=1}^{2\ell+1}\abs{Y_{\ell,k}(x)}^2
  =\frac{2\ell+1}{4\pi},
  \qquad x\in\Sph.
\end{equation}
For \(v\in L^2(\Sph)\), its Fourier--Laplace coefficients are
$\widehat v_{\ell,k}
  =\int_{\Sph}v(x)Y_{\ell,k}(x)\dd\omega(x)$.
For \(r\in\mathbb R\), we use the Sobolev space defined through the
Laplace--Beltrami spectrum, with norm
\begin{equation}\label{eq:sobolev}
  \|v\|_{H^r(\Sph)}^2
  =\sum_{\ell=0}^{\infty}
    \bigl(1+\ell(\ell+1)\bigr)^r
    \sum_{k=1}^{2\ell+1}\abs{\widehat v_{\ell,k}}^2.
\end{equation}
Since \(1+\ell(\ell+1)\asymp(1+\ell)^2\), this norm is equivalent to the
coefficient norm with weights \((1+\ell)^{2r}\) used below.  In particular,
\(H^r(\Sph)\hookrightarrow C(\Sph)\) continuously for \(r>1\). We denote by \(\lambda_\ell:=1+\ell(\ell+1)\).

\begin{proposition}[Multiplier mappings]
Under Assumption~\ref{ass:operator}, for every \(a\in\mathbb R\),
\(\Id-\Top\) is boundedly invertible on \(H^a(\Sph)\).  With
$\Rop:=(\Id-\Top)^{-1}$ and $\Bop:=\Top\Rop$,
we have
\[
  \Top Y_{\ell,k}=\mu_\ell Y_{\ell,k},\qquad
  \Rop Y_{\ell,k}=\frac{1}{1-\mu_\ell}Y_{\ell,k},\qquad
  \Bop Y_{\ell,k}=\frac{\mu_\ell}{1-\mu_\ell}Y_{\ell,k}.
\]
Moreover,
\[
\begin{aligned}
  &\Top:H^a(\Sph)\to H^{a+\beta}(\Sph),
  &&\|\Top\|_{H^a\to H^{a+\beta}}\le C_{\Top},\\
  &\Rop:H^a(\Sph)\to H^a(\Sph),
  &&\|\Rop\|_{H^a\to H^a}\le\delta^{-1},\\
  &\Bop:H^a(\Sph)\to H^{a+\beta}(\Sph),
  &&\|\Bop\|_{H^a\to H^{a+\beta}}\le C_{\Top}\delta^{-1}.
\end{aligned}
\]\end{proposition}
\begin{proof}
 Since
\(\lambda_\ell\leq(1+\ell)^2\), Assumption~\ref{ass:operator} gives
$|\mu_\ell|^2
  \leq C_{\Top}^2(1+\ell)^{-2\beta}
  \leq C_{\Top}^2\lambda_\ell^{-\beta}$.
By \eqref{eq:funk-hecke-intro},
\((\widehat{\Top v})_{\ell,k}=\mu_\ell\widehat v_{\ell,k}\).  Hence
\begin{equation*}
  \|\Top v\|_{H^{a+\beta}}^2=\sum_{\ell=0}^{\infty}
    \lambda_\ell^{a+\beta}|\mu_\ell|^2
    \sum_{k=1}^{2\ell+1}|\widehat v_{\ell,k}|^2 \leq C_{\Top}^2
    \sum_{\ell=0}^{\infty}
    \lambda_\ell^a
    \sum_{k=1}^{2\ell+1}|\widehat v_{\ell,k}|^2
   =C_{\Top}^2\|v\|_{H^a}^2.
\end{equation*}
Since
\(
  (\Id-\Top)Y_{\ell,k}=(1-\mu_\ell)Y_{\ell,k}
\),
consider the operator
\[
  \Rop v
  :=\sum_{\ell=0}^{\infty}\sum_{k=1}^{2\ell+1}
    \frac{\widehat v_{\ell,k}}{1-\mu_\ell}Y_{\ell,k}.
\]
The nonresonance condition yields
\(
  \|\Rop v\|_{H^a}^2
  \leq \delta^{-2}\|v\|_{H^a}^2.
\)
Thus \(\Rop\) is bounded on \(H^a(\Sph)\), and coefficientwise
multiplication gives
\(
  (\Id-\Top)\Rop=\Rop(\Id-\Top)=\Id.
\)
Hence \(\Rop=(\Id-\Top)^{-1}\).
Finally, \(\Bop=\Top\Rop\) has multiplier
\(\mu_\ell/(1-\mu_\ell)\), and therefore
\begin{equation*}
  \|\Bop v\|_{H^{a+\beta}}^2
  =
  \sum_{\ell=0}^{\infty}
  \frac{\lambda_\ell^{a+\beta}|\mu_\ell|^2}
       {|1-\mu_\ell|^2}
  \sum_{k=1}^{2\ell+1}|\widehat v_{\ell,k}|^2 \leq
  C_{\Top}^2\delta^{-2}\|v\|_{H^a}^2,
\end{equation*}
which completes the proof.
\end{proof}
The definitions also yield the resolvent identities
\begin{equation}\label{eq:B-identities}
  \Rop=\Id+\Bop,
  \qquad
  \Bop=\Top+\Top\Bop=\Top+\Bop\Top.
\end{equation}
Indeed, \((\Id-\Top)\Rop=\Id\) gives
\(\Rop-\Id=\Top\Rop=\Bop\).  Since \(\Top\) and \(\Rop\) commute,
multiplying \(\Rop=\Id+\Bop\) by \(\Top\) gives the remaining identities.

Let \(\gamma\ge0\) and \((m_\ell)_{\ell\ge0}\) be a scalar sequence
satisfying $|m_\ell|\le C_{\Mop}(1+\ell)^{-\gamma}$
for $\ell\in\Nzero$.
The associated spherical multiplier operator \(\Mop\) is defined by
\begin{equation}\label{equ:operatorM}
  \Mop v
  :=\sum_{\ell=0}^{\infty}\sum_{k=1}^{2\ell+1}
    m_\ell\widehat v_{\ell,k}Y_{\ell,k},
  \qquad v\in L^2(\Sph).
\end{equation}
The series converges in \(L^2(\Sph)\), with
\(\Mop Y_{\ell,k}=m_\ell Y_{\ell,k}\).  In particular, \(\Mop\) is bounded
on \(L^2(\Sph)\), and \(\gamma\) is its guaranteed algebraic smoothing order.

\begin{lemma}[Filtered spectral tail]\label{lem:tail}
If $v\in H^r(\Sph)$ and $r+\gamma>1$, then, for $n\geq1$,
\begin{equation}\label{eq:tail-bound}
  \|\Mop(\Id-\Pn)v\|_{L^\infty(\Sph)}
  \leq C n^{-(r+\gamma-1)}\|v\|_{H^r(\Sph)}.
\end{equation}
\end{lemma}

\begin{proof}
By \eqref{equ:operatorM},
\[
  \Mop(\Id-\Pn)v
  =\sum_{\ell>n}\sum_{k=1}^{2\ell+1}
    m_\ell\widehat v_{\ell,k}Y_{\ell,k}.
\]
Hence, by Cauchy--Schwarz, the multiplier bound, and the equivalence of the
Sobolev norm with the coefficient norm weighted by \((1+\ell)^{2r}\),
\begin{align*}
 \abs{\Mop(\Id-\Pn)v(x)}
 &\leq
 C\|v\|_{H^r}
 \left(
 \sum_{\ell>n}(1+\ell)^{-2(r+\gamma)}
 \sum_{k=1}^{2\ell+1}\abs{Y_{\ell,k}(x)}^2
 \right)^{1/2}  \\
 &=\frac{C}{\sqrt{4\pi}}\|v\|_{H^r}
 \left(
 \sum_{\ell>n}(2\ell+1)(1+\ell)^{-2(r+\gamma)}
 \right)^{1/2} \\
 &\leq
 C\|v\|_{H^r}
 \left(
 \sum_{\ell>n}(1+\ell)^{1-2(r+\gamma)}
 \right)^{1/2},
\end{align*}
where the equality follows from the addition formula
\eqref{eq:addition}.  Since \(r+\gamma>1\),
\[
  \sum_{\ell>n}(1+\ell)^{1-2(r+\gamma)}
  \leq
  \int_n^\infty(1+x)^{1-2(r+\gamma)}\,\dd x
  =
  \frac{(1+n)^{2-2(r+\gamma)}}{2(r+\gamma)-2}.
\]
Taking square roots and then the supremum over \(x\in\Sph\) gives
\eqref{eq:tail-bound}.
\end{proof}

The exponent in Lemma~\ref{lem:tail} is sharp when the multiplier has exact
algebraic order.

\begin{proposition}[Sharp filtered-tail rate]\label{prop:sharp-tail}
Suppose that the real multiplier sequence in \eqref{equ:operatorM} satisfies
\begin{equation}\label{equ:twosided}
  c_{\Mop}(1+\ell)^{-\gamma}
  \leq |m_\ell|
  \leq C_{\Mop}(1+\ell)^{-\gamma},
  \qquad \ell\geq\ell_0,
\end{equation}
for some \(c_{\Mop},C_{\Mop}>0\).  If \(r+\gamma>1\), then, for all
sufficiently large \(n\),
\begin{equation}\label{eq:sharp-tail-operator-norm}
  \|\Mop(\Id-\Pn)\|_{H^r(\Sph)\to L^\infty(\Sph)}
  \asymp n^{-(r+\gamma-1)}.
\end{equation}
\end{proposition}

\begin{proof}
For sufficiently large \(n\), the upper bound follows from
Lemma~\ref{lem:tail}.  For the reverse bound, we fix \(x_0\in\Sph\) and define
\[
  A_n
  :=\sum_{n<\ell\leq2n}
    \lambda_\ell^{-r}|m_\ell|^2
    \sum_{k=1}^{2\ell+1}|Y_{\ell,k}(x_0)|^2.
\]
For \(n\geq\ell_0\), the addition formula \eqref{eq:addition} and the bound \eqref{equ:twosided}
give
\(
  A_n\asymp n^{2-2(r+\gamma)}.
\)
Define \(v_n\) by $(\widehat{v_n})_{\ell,k}
  =A_n^{-1/2}\lambda_\ell^{-r}m_\ell Y_{\ell,k}(x_0)$
  for $n<\ell\leq2n$,
with all other coefficients equal to zero.  Then
\[
  \|v_n\|_{H^r}^2
  =A_n^{-1}
    \sum_{n<\ell\leq2n}
      \lambda_\ell^{-r}|m_\ell|^2
      \sum_{k=1}^{2\ell+1}|Y_{\ell,k}(x_0)|^2
  =1.
\]
Moreover,
\(
  \Mop(\Id-\Pn)v_n(x_0)
  =A_n^{1/2}
  \asymp n^{-(r+\gamma-1)}.
\)
Hence
\(
  \|\Mop(\Id-\Pn)\|_{H^r\to L^\infty}
  \gtrsim n^{-(r+\gamma-1)},
\)
which, together with Lemma~\ref{lem:tail}, proves
\eqref{eq:sharp-tail-operator-norm}.
\end{proof}

\begin{lemma}[Filtered low-degree error]\label{lem:low}
If \(g\in\mathbb P_n\) with \(n\ge2\), then
\begin{equation}\label{eq:low-bound}
  \|\Mop g\|_{L^\infty(\Sph)}
  \leq C\rho_\gamma(n)\|g\|_{L^2(\Sph)},
\end{equation}
where
\[
  \rho_\gamma(n)=
  \begin{cases}
    n^{1-\gamma}, & 0\leq\gamma<1,\\
    (\log(n+1))^{1/2}, & \gamma=1,\\
    1, & \gamma>1.
  \end{cases}
\]
\end{lemma}

\begin{proof}
Since \(g\in\mathbb P_n\),
\[
  \Mop g(x)
  =\sum_{\ell=0}^{n}\sum_{k=1}^{2\ell+1}
    m_\ell\widehat g_{\ell,k}Y_{\ell,k}(x).
\]
For fixed \(x\in\Sph\), the Cauchy--Schwarz inequality, Parseval's identity, and the
addition formula \eqref{eq:addition} give
\begin{align*}
 |\Mop g(x)|
 &\leq
 \|g\|_{L^2(\Sph)}
 \left(
   \sum_{\ell=0}^{n}|m_\ell|^2
   \sum_{k=1}^{2\ell+1}|Y_{\ell,k}(x)|^2
 \right)^{1/2}=\frac{\|g\|_{L^2(\Sph)}}{\sqrt{4\pi}}
 \left(
   \sum_{\ell=0}^{n}(2\ell+1)|m_\ell|^2
 \right)^{1/2} \\
 &\leq
 C\|g\|_{L^2(\Sph)}
 \left(
   \sum_{\ell=0}^{n}(1+\ell)^{1-2\gamma}
 \right)^{1/2}.
\end{align*}
The last sum satisfies
\[
  \sum_{\ell=0}^{n}(1+\ell)^{1-2\gamma}
  \leq
  \begin{cases}
    Cn^{2-2\gamma}, & 0\leq\gamma<1,\\
    C\log(n+1), & \gamma=1,\\
    C, & \gamma>1,
  \end{cases}
\]
by the integral test, with convergence of the infinite series when
\(\gamma>1\).  Taking square roots and the supremum over \(x\in\Sph\)
proves \eqref{eq:low-bound}.
\end{proof}

\section{Quadrature-free Sloan iteration}
\label{sec:setting}

The Sloan error identity \eqref{eq:error-intro} reduces the analysis to the
action of \(\Top^q\) on the initial error.  In the absence of quadrature, the
Galerkin and degenerate-kernel errors are supported entirely in degrees
\(\ell>n\).  Since \(\Top\) and \(\Pn\) are diagonal in the same
spherical harmonic basis, both starting approximations admit explicit spectral
representations, and their Sloan errors follow directly from
Lemma~\ref{lem:tail}.

\subsection{Galerkin methods}

For \(f\in L^2(\Sph)\), the Galerkin solution
\(u_n^G\in\mathbb P_n\) is defined by
\begin{equation}\label{eq:galerkin}
  \left\langle(\Id-\Top)u_n^G,p\right\rangle_{L^2(\Sph)}
  =
  \left\langle f,p\right\rangle_{L^2(\Sph)}
  \qquad \forall\, p\in\mathbb P_n.
\end{equation}

\begin{proposition}[Galerkin solution]
Under Assumption~\ref{ass:operator}, problem~\eqref{eq:galerkin} has a unique
solution satisfying
\begin{equation}\label{eq:G-closed-form}
  u_n^G=\Rop\Pn f=\Pn\Rop f=\Pn u.
\end{equation}
Its Fourier--Laplace coefficients are
\begin{equation*}
  (\widehat{u_n^G})_{\ell,k}=
  \begin{cases}
    \widehat f_{\ell,k}/(1-\mu_\ell), & 0\leq\ell\leq n,\\
    0, & \ell>n.
  \end{cases}
\end{equation*}
Hence the Galerkin error is the spectral tail
\begin{equation}\label{eq:G-projection}
  u-u_n^G
  =\Rop(\Id-\Pn)f
  =(\Id-\Pn)\Rop f.
\end{equation}
\end{proposition}

\begin{proof}
Since \(u_n^G\in\mathbb P_n\) and
\(\Top\mathbb P_n\subseteq\mathbb P_n\), \eqref{eq:galerkin} gives
$\Pn (\Id-\Top)u_n^G=
  (\Id-\Top)u_n^G=\Pn f$.
Thus \(u_n^G=\Rop\Pn f=\Pn\Rop f=\Pn u\), since \(\Rop\) commutes with
\(\Pn\).  The coefficient formula and the error representation follow immediately.
\end{proof}

By \eqref{eq:G-projection}, the initial Galerkin error is supported entirely
in degrees \(\ell>n\).  Each Sloan step therefore applies one further factor
of \(\mu_\ell\) to the spectral tail.

\begin{theorem}[Galerkin Sloan iterates]
Let Assumption~\ref{ass:operator} hold, let \(q\in\Nzero\), and suppose that
\(f\in H^r(\Sph)\) with \(r+q\beta>1\).  Then
$u-\mathcal S_q(u_n^G)
  =\Mop_{G,q}(\Id-\Pn)f$,
where \(\Mop_{G,q}\) is the spherical multiplier
\begin{equation}\label{eq:G-multiplier}
  \Mop_{G,q}Y_{\ell,k}
  =m_\ell^{G,q}Y_{\ell,k},
  \qquad
  m_\ell^{G,q}
  =\frac{\mu_\ell^q}{1-\mu_\ell}.
\end{equation}
Moreover,
\begin{equation}\label{eq:G-rate}
  \|u-\mathcal S_q(u_n^G)\|_{L^\infty}
  \leq
  C n^{-(r+q\beta-1)}\|f\|_{H^r}.
\end{equation}
\end{theorem}

\begin{proof}
By the Sloan error identity \eqref{eq:error-intro} and
\eqref{eq:G-projection},
\[
  u-\mathcal S_q(u_n^G)
  =\Top^q\Rop(\Id-\Pn)f
  =\Mop_{G,q}(\Id-\Pn)f.
\]
Assumption~\ref{ass:operator} gives
\[
  |m_\ell^{G,q}|
  =\frac{|\mu_\ell|^q}{|1-\mu_\ell|}
  \leq C(1+\ell)^{-q\beta}.
\]
Lemma~\ref{lem:tail} with \(\gamma=q\beta\) then yields
\eqref{eq:G-rate}.
\end{proof}
\begin{remark}[The quadrature-free Galerkin rate]
Each Galerkin Sloan step contributes one additional factor \(\mu_\ell\) on
the spectral tail and therefore adds \(\beta\) orders of spectral smoothing.
The loss of one power in \eqref{eq:G-rate} results from passage to the
uniform norm on the two-dimensional sphere.  In \(L^2\), the corresponding
estimate is
\(
  \|u-\mathcal S_q(u_n^G)\|_{L^2(\Sph)}
  \leq Cn^{-(r+q\beta)}\|f\|_{H^r(\Sph)}.
\)
For \(q=0\), \eqref{eq:G-rate} reduces to the standard uniform projection
rate \(O(n^{-(r-1)})\) for \(r>1\). 
\end{remark}

\subsection{Degenerate-kernel methods}

Since \(\Top\) and \(\Pn\) are diagonal in the same spherical harmonic basis,
we define
$\Top_n:=\Top\Pn=\Pn\Top=\Pn\Top\Pn$.
Then \(\operatorname{ran}(\Top_n)\subseteq\mathbb P_n\), so
\(\operatorname{rank}(\Top_n)\leq (n+1)^2\).  More explicitly,
\begin{equation}\label{eq:degenerate-kernel}
  (\Top_n v)(x)=\int_{\Sph}K_n(x,y)v(y)\dd\omega(y),
  \qquad
  K_n(x,y)=\sum_{\ell=0}^n\sum_{k=1}^{2\ell+1}
  \mu_\ell Y_{\ell,k}(x)Y_{\ell,k}(y).
\end{equation}
Since \(K_n\) is a finite sum of separable terms, it is a degenerate kernel,
and \(\Top_n\) is the corresponding finite-rank approximation of \(\Top\)
used in the classical degenerate-kernel method
\cite{Sloan1976Degenerate}.  We define the degenerate-kernel solution \(u_n^D\) by
\begin{equation}\label{eq:degenerate}
  (\Id-\Top\Pn)u_n^D=f.
\end{equation}
Unlike the Galerkin solution, $u_n^D$ generally does not belong to
$\mathbb P_n$.

\begin{proposition}[Degenerate-kernel solution]
Under Assumption~\ref{ass:operator}, \(\Id-\Top\Pn\) is invertible on
\(L^2(\Sph)\), with
$(\Id-\Top\Pn)^{-1}=\Id+\Bop\Pn$.
Hence the degenerate-kernel approximation \eqref{eq:degenerate} is uniquely
defined and satisfies
\begin{equation}\label{eq:D-formula}
  u_n^D=f+\Bop\Pn f,
  \qquad
  \Pn u_n^D=\Rop\Pn f=\Pn u.
\end{equation}
Its Fourier--Laplace coefficients are
\begin{equation}\label{equ:dkcoeff}
  (\widehat{u_n^D})_{\ell,k}=
  \begin{cases}
    \widehat f_{\ell,k}/(1-\mu_\ell), & 0\leq\ell\leq n,\\
    \widehat f_{\ell,k}, & \ell>n.
  \end{cases}
\end{equation}
Moreover,
\begin{equation}\label{eq:D-error}
  u-u_n^D=\Bop(\Id-\Pn)f.
\end{equation}
\end{proposition}

\begin{proof}
Since \(\Top\), \(\Bop\), and \(\Pn\) commute and \(\Pn^2=\Pn\),
the identities \eqref{eq:B-identities} give
\begin{align*}
  (\Id-\Top\Pn)(\Id+\Bop\Pn)
  &=\Id+(\Bop-\Top-\Top\Bop)\Pn=\Id,\\
  (\Id+\Bop\Pn)(\Id-\Top\Pn)
  &=\Id+(\Bop-\Top-\Bop\Top)\Pn=\Id,
\end{align*}
yielding the inverse. Hence
\(
  u_n^D=(\Id+\Bop\Pn)f=f+\Bop\Pn f.
\)
Applying \(\Pn\) and using \(\Rop=\Id+\Bop\) yields
$\Pn u_n^D
  =\Pn f+\Bop\Pn f
  =\Rop\Pn f
  =\Pn\Rop f
  =\Pn u$.
The formula \eqref{equ:dkcoeff} follows immediately.  Finally, 
\eqref{eq:D-error} follows from \(u=\Rop f=f+\Bop f\).
\end{proof}

\begin{theorem}[Degenerate-kernel Sloan iterates]
Let Assumption~\ref{ass:operator} hold, let $q\in\Nzero$, and suppose
$f\in H^r(\Sph)$ with $r+(q+1)\beta>1$.  Then
\begin{equation}
  u-\mathcal S_q(u_n^D)=\Mop_{D,q}(\Id-\Pn)f,
\end{equation}
where
\begin{equation}\label{eq:D-multiplier}
  \Mop_{D,q}Y_{\ell,k}
  =m_\ell^{D,q}Y_{\ell,k},
  \qquad
  m_\ell^{D,q}=\frac{\mu_\ell^{q+1}}{1-\mu_\ell}.
\end{equation}
Consequently,
\begin{equation}\label{eq:D-rate}
  \|u-\mathcal S_q(u_n^D)\|_{L^\infty}
  \leq C n^{-(r+(q+1)\beta-1)}\|f\|_{H^r}.
\end{equation}
\end{theorem}

\begin{proof}
Apply the Sloan error identity \eqref{eq:error-intro} to
\eqref{eq:D-error}.  The resulting operator $\Top^q\Bop$ has multiplier
$m_\ell^{D,q}$, which is bounded by
$C(1+\ell)^{-(q+1)\beta}$.  Lemma~\ref{lem:tail} gives the result.
\end{proof}
\begin{remark}[The quadrature-free degenerate-kernel rate]
The initial error \eqref{eq:D-error} already contains one smoothing factor
through \(\Bop\).  Accordingly,
\(
  \|u-\mathcal S_q(u_n^D)\|_{L^2(\Sph)}
  \leq Cn^{-(r+(q+1)\beta)}\|f\|_{H^r(\Sph)}.
\)
As in the Galerkin case, passage to the uniform norm on $\Sph$ costs one power of \(n\).
\end{remark}

\begin{remark}[One-layer relation]\label{rem:one-layer}
For every \(q\in\Nzero\),
\(
  u-\mathcal S_q(u_n^D)
  =u-\mathcal S_{q+1}(u_n^G).
\)
Indeed, \eqref{eq:G-multiplier} and \eqref{eq:D-multiplier} give
\(m_\ell^{D,q}=m_\ell^{G,q+1}\) for every \(\ell\).
\end{remark}
\section{Sloan iteration with exact quadrature}
\label{sec:discrete}

Replacing \(\Pn\) by the hyperinterpolation operator \(\Ln\) leads to two
quadrature-based schemes: discrete Galerkin methods and product-integration
Nystr\"om methods.  Under polynomial exactness, the representations from
Section~\ref{sec:setting} remain valid, but
$f-\Ln f=(\Id-\Pn)f+(\Pn-\Ln)f$
now contains a spectral tail and a low-degree aliasing term.  We first
estimate the aliasing contribution, then derive the two discrete starting
methods and their Sloan errors, and finally establish the
recovery--saturation threshold.
Throughout this section, let \(\{\mathcal Q_{N_n}\}_{n\ge1}\) be a family of
positive quadrature rules,
\[
  \mathcal Q_{N_n}(g)
  =\sum_{j=1}^{N_n}w_{j,n}g(x_{j,n}),
  \qquad x_{j,n}\in\Sph,\quad w_{j,n}>0.
\]
Assume that \(\mathcal Q_{N_n}\) has algebraic degree of precision
\(t_n\ge2n\), that is,
\begin{equation}\label{eq:quadrature}
  \mathcal Q_{N_n}(p)
  =\int_{\Sph}p(x)\dd\omega(x),
  \qquad p\in\mathbb P_{t_n}.
\end{equation}
The condition \(t_n\ge2n\) is the standard requirement for degree-\(n\)
hyperinterpolation; see, e.g.,
\cite{an2025path,MR2274179,sloan1995polynomial,SloanWomersley2000}.
For fixed \(n\), we denote by \(\mathcal Q=\mathcal Q_{N_n}\) for simplicity. Let \(\Ln\) denote
the associated hyperinterpolation operator, and set
$\Qip{v}{z}:=\mathcal Q_{N_n}(vz)$
and
$s_n:=t_n-n$.
Then \(s_n\ge n\), and the excess exactness beyond \(2n\) enters the aliasing
estimate through \(s_n\).  Exactness on \(\mathbb P_{2n}\) also implies
\(\Ln p=p\) for every \(p\in\mathbb P_n\); see, e.g.,
\cite{sloan1995polynomial}.

\begin{lemma}[Aliasing estimate]\label{lem:alias-source}
Under \eqref{eq:quadrature}, if \(v\in H^r(\Sph)\) with \(r>1\), then
\begin{equation}\label{eq:alias-sobolev}
  \|\Pn v-\Ln v\|_{L^2}
  \leq C s_n^{-r}\|v\|_{H^r}.
\end{equation}
\end{lemma}

\begin{proof}
Let \(m_n=\lfloor t_n/2\rfloor\), and let
\(\mathcal L_{m_n}v\) be the degree-\(m_n\) hyperinterpolant associated with
\(\mathcal Q_{N_n}\).  Since \(2m_n\leq t_n\), the estimate for
hyperinterpolation in \cite[Theorem~1.2 and Section 2.5]{LuWang2023} gives
$\|v-\mathcal L_{m_n}v\|_{L^2}
  \leq C m_n^{-r}\|v\|_{H^r}$.
Since \(t_n\geq2n\), we have \(m_n\geq n\), and the coefficients of
\(\mathcal L_{m_n}v\) through degree \(n\) coincide with those of \(\Ln v\).
Hence
\(
  \Ln v=\Pn\mathcal L_{m_n}v,
\)
and therefore
$\|\Pn v-\Ln v\|_{L^2}
  =\|\Pn(v-\mathcal L_{m_n}v)\|_{L^2}
  \leq \|v-\mathcal L_{m_n}v\|_{L^2}
  \leq C m_n^{-r}\|v\|_{H^r}$.
Finally, \(t_n/2\leq s_n=t_n-n\leq t_n\) and
\(m_n\asymp t_n\), so \(m_n\asymp s_n\), which proves
\eqref{eq:alias-sobolev}.
\end{proof}

\subsection{Discrete Galerkin methods}

For $f\in C(\Sph)$, the discrete Galerkin solution
$u_{n,\mathcal Q}^G\in\mathbb{P}_n$ is defined by
\begin{equation}\label{eq:discrete-galerkin}
  \Qip{(\Id-\Top)u_{n,\mathcal Q}^G}{p}=\Qip{f}{p},
  \qquad p\in\mathbb{P}_n.
\end{equation}
Thus the same quadrature rule is used for the Galerkin inner products and for
the hyperinterpolation operator $\Ln$.

\begin{proposition}[Discrete Galerkin solution]
\label{prop:QG-formula}
Under Assumption~\ref{ass:operator} and \eqref{eq:quadrature},
problem~\eqref{eq:discrete-galerkin} has the unique solution
\begin{equation}\label{eq:QG-formula}
  u_{n,\mathcal Q}^G=\Rop\Ln f.
\end{equation}
Consequently,
\begin{equation}\label{eq:QG-direct-error}
  u-u_{n,\mathcal Q}^G=\Rop(f-\Ln f).
\end{equation}
\end{proposition}

\begin{proof}
Since \(u_{n,\mathcal Q}^G\in\mathbb P_n\) and
\(\Top\mathbb P_n\subseteq\mathbb P_n\), we have
\(w:=(\Id-\Top)u_{n,\mathcal Q}^G\in\mathbb P_n\).
Exactness on \(\mathbb P_{2n}\) and the definition of \(\Ln\) give
$\Qip{w}{p}=\langle w,p\rangle_{L^2}$ and
$\Qip{f}{p}=\langle \Ln f,p\rangle_{L^2}$ for
$p\in\mathbb P_n$.
Thus \eqref{eq:discrete-galerkin} implies
\(\langle w-\Ln f,p\rangle_{L^2}=0\) for all \(p\in\mathbb P_n\), and hence
\(w=\Ln f\).  Therefore
$u_{n,\mathcal Q}^G=\Rop\Ln f$,
which also proves existence and uniqueness.  Subtracting from
\(u=\Rop f\) gives \eqref{eq:QG-direct-error}.
\end{proof}

\subsection{Product-integration Nystr\"om methods}

For a weakly singular kernel, direct quadrature of
\(\Phi(x\cdot y)v(y)\) may require the singular value \(\Phi(1)\) when the
target \(x\) coincides with a quadrature point.  Product integration avoids
this difficulty by replacing \(v\) with its hyperinterpolant and integrating
the kernel against each spherical harmonic analytically via the
Funk--Hecke formula; see, e.g., \cite{MR5025409,Sloan1980ProductIntegration}.  Thus, we define
\(
  \Top_n^{\mathrm{PI}}v:=\Top\Ln v.
\)
Expanding \(\Ln v\) and its discrete coefficients gives the Nystr\"om form
\begin{equation*}
\begin{aligned}
  (\Top\Ln v)(x)
  &=\sum_{\ell=0}^{n}\sum_{k=1}^{2\ell+1}
    \mu_\ell\Qip{v}{Y_{\ell,k}}Y_{\ell,k}(x) =\sum_{j=1}^{N_n}w_{j,n}v(x_{j,n})
    \sum_{\ell=0}^{n}\sum_{k=1}^{2\ell+1}
    \mu_\ell Y_{\ell,k}(x)Y_{\ell,k}(x_{j,n}) \\
  &=\sum_{j=1}^{N_n}W_{j,n}(x)v(x_{j,n}),
\end{aligned}
\end{equation*}
where
\(
  W_{j,n}(x):=w_{j,n}K_n(x,x_{j,n})\).
Here the first equality uses the Funk--Hecke formula
\eqref{eq:funk-hecke-intro}, while the second uses the finite-rank kernel
\(K_n\) from \eqref{eq:degenerate-kernel}.  The weights \(W_{j,n}(x)\)
therefore incorporate the kernel analytically, and quadrature is applied
only to the coefficients of \(v\).  In contrast, direct Nystr\"om
quadrature would involve the possibly singular values
\(w_j\Phi(x\cdot x_j)\).

For \(f\in H^r(\Sph)\), \(r>1\), the product-integration Nystr\"om
solution \(u_{n,\mathcal Q}^N\) is defined by
\begin{equation}
  (\Id-\Top\Ln)u_{n,\mathcal Q}^N=f.
\end{equation}

\begin{proposition}[Product-integration Nystr\"om solution]
\label{prop:N-formula}
Under Assumption~\ref{ass:operator} and \eqref{eq:quadrature},
\(\Id-\Top\Ln\) is invertible on \(H^r(\Sph)\) for every \(r>1\), with
\begin{equation}\label{eq:N-inverse}
  (\Id-\Top\Ln)^{-1}=\Id+\Bop\Ln.
\end{equation}
Consequently,
\begin{equation}\label{eq:N-formula}
  u_{n,\mathcal Q}^N=f+\Bop\Ln f,
  \qquad
  u-u_{n,\mathcal Q}^N=\Bop(f-\Ln f).
\end{equation}
\end{proposition}

\begin{proof}
Exactness on \(\mathbb P_{2n}\) implies that \(\Ln\) reproduces
\(\mathbb P_n\).  Since \(\operatorname{ran}(\Ln)=\mathbb P_n\) and
\(\Top\) and \(\Bop\) preserve \(\mathbb P_n\), we have $\Ln\Top\Ln=\Top\Ln$
and $\Ln\Bop\Ln=\Bop\Ln$. Using \eqref{eq:B-identities}, we therefore obtain
\begin{align*}
  (\Id-\Top\Ln)(\Id+\Bop\Ln)
  &=\Id+(\Bop-\Top-\Top\Bop)\Ln=\Id,\\
  (\Id+\Bop\Ln)(\Id-\Top\Ln)
  &=\Id+(\Bop-\Top-\Bop\Top)\Ln=\Id.
\end{align*}
For \(r>1\), the embedding \(H^r(\Sph)\hookrightarrow C(\Sph)\) makes
\(\Ln:H^r(\Sph)\to\mathbb P_n\) bounded, so \(\Id+\Bop\Ln\) is a bounded
two-sided inverse on \(H^r(\Sph)\).  This proves \eqref{eq:N-inverse} and
hence
\(
  u_{n,\mathcal Q}^N
  =(\Id+\Bop\Ln)f
  =f+\Bop\Ln f.
\)
Finally, since \(u=\Rop f=f+\Bop f\),
it follows \eqref{eq:N-formula}.
\end{proof}

\subsection{Sloan error decomposition and upper bounds}
The two discrete starting methods can now be treated in a unified form.
For \(\star\in\{G,N\}\), we set
$\varepsilon_G=0$ and $\varepsilon_N=1$,
and define the spherical multiplier \(\Mop_{\star,q}\) by
\begin{equation}\label{eq:paired-notation}
  \Mop_{\star,q}Y_{\ell,k}
  =
  \frac{\mu_\ell^{q+\varepsilon_\star}}{1-\mu_\ell}
  Y_{\ell,k}.
\end{equation}
Let \(u_{n,\mathcal Q}^{\star}\) denote the corresponding discrete Galerkin
or product-integration Nystr\"om solution.  Then, for
\(f\in C(\Sph)\) and \(q\in\Nzero\), the Sloan error identity
\eqref{eq:error-intro}, together with
\eqref{eq:QG-direct-error} and \eqref{eq:N-formula}, gives
\begin{equation}\label{eq:star-error}
  u-\mathcal S_q(u_{n,\mathcal Q}^{\star})
  =\Mop_{\star,q}(f-\Ln f).
\end{equation}
Using the decomposition \eqref{eq:split-intro},
\begin{equation}\label{eq:star-split}
  u-\mathcal S_q(u_{n,\mathcal Q}^{\star})
  =
  \Mop_{\star,q}(\Id-\Pn)f
  +
  \Mop_{\star,q}(\Pn f-\Ln f).
\end{equation}
The two terms are the filtered spectral tail and filtered aliasing term,
respectively.

\begin{theorem}[Sloan estimates for the discrete methods]\label{thm:Q}
Assume Assumption~\ref{ass:operator} and \eqref{eq:quadrature}.  Let
\(q\in\Nzero\), \(\star\in\{G,N\}\), and set
\(
  \gamma_\star=(q+\varepsilon_\star)\beta.
\)
If \(f\in H^r(\Sph)\) with \(r>1\), then
\begin{equation}\label{eq:Q-main-sobolev}
  \|u-\mathcal S_q(u_{n,\mathcal Q}^{\star})\|_{L^\infty}
  \leq
  C\left(
    n^{-(r+\gamma_\star-1)}
    +\rho_{\gamma_\star}(n)s_n^{-r}
  \right)\|f\|_{H^r}.
\end{equation}
\end{theorem}

\begin{proof}
By Assumption~\ref{ass:operator},
$\left|
    {\mu_\ell^{q+\varepsilon_\star}}/{(1-\mu_\ell)}
  \right|
  \leq C(1+\ell)^{-\gamma_\star}$.
Applying Lemma~\ref{lem:tail} to the first term in
\eqref{eq:star-split} gives
$\|\Mop_{\star,q}(\Id-\Pn)f\|_{L^\infty}
  \leq
  Cn^{-(r+\gamma_\star-1)}\|f\|_{H^r}$.
Since \(\Pn f-\Ln f\in\mathbb P_n\), Lemmas~\ref{lem:low}
and~\ref{lem:alias-source} give
$\|\Mop_{\star,q}(\Pn f-\Ln f)\|_{L^\infty}
  \leq
  C\rho_{\gamma_\star}(n)s_n^{-r}\|f\|_{H^r}$.
Combining the two estimates proves \eqref{eq:Q-main-sobolev}.
\end{proof}

\begin{corollary}
In particular, \eqref{eq:Q-main-sobolev} becomes
\[
\begin{aligned}
 \|u-\mathcal S_q(u_{n,\mathcal Q}^{G})\|_{L^\infty}
 &\leq C\left(
 n^{-(r+q\beta-1)}
 +\rho_{q\beta}(n)s_n^{-r}\right)\|f\|_{H^r},\\
 \|u-\mathcal S_q(u_{n,\mathcal Q}^{N})\|_{L^\infty}
 &\leq C\left(
 n^{-(r+(q+1)\beta-1)}
 +\rho_{(q+1)\beta}(n)s_n^{-r}\right)\|f\|_{H^r}.
\end{aligned}
\]
\end{corollary}

\begin{remark}[One-layer relation after quadrature]\label{rem:onelayer}
The quadrature-free one-layer relation in
Remark~\ref{rem:one-layer} persists after discretization:
for every \(q\in\Nzero\), $u-\mathcal S_q(u_{n,\mathcal Q}^N)
  =
  u-\mathcal S_{q+1}(u_{n,\mathcal Q}^G)$.
Indeed, both sides equal
\(\Top^{q+1}\Rop(f-\Ln f)\) by \eqref{eq:star-error}.
\end{remark}

\subsection{Recovery threshold and sharp saturation}

The constant harmonic transfers scalar quadrature error to a coefficient of
the Sloan error.  Combined with Theorem~\ref{thm:Q}, this yields matching
worst-case bounds.  For the \(n\)th quadrature rule, define the worst-case
error
\begin{equation*}
 E_{n,q}^{\star}(\mathcal Q_{N_n})
 :=\sup_{0\ne f\in H^r(\Sph)}
 \frac{\|u-\mathcal S_q
   (u_{n,\mathcal Q_{N_n}}^{\star})\|_{L^\infty}}
      {\|f\|_{H^r}}.
\end{equation*}
\begin{lemma}[Constant-mode lower bound]
\label{lem:alias-lower}
Let Assumption~\ref{ass:operator} hold, let \(r>1\), \(q\in\Nzero\), and
\(\star\in\{G,N\}\).  Suppose that \(\mathcal Q_{N_n}\) satisfies
\eqref{eq:quadrature} and that
\begin{equation}\label{eq:nonzero-constant-mode}
  m_{0}^{\star,q}
  =\frac{\mu_0^{q+\varepsilon_\star}}{1-\mu_0}\ne0.
\end{equation}
Then $E_{n,q}^{\star}(\mathcal Q_{N_n})
  \geq c_r |m_0^{\star,q}|N_n^{-r/2}$,
where \(c_r>0\) is independent of \(n\), \(t_n\), \(N_n\), and
\(\mathcal Q_{N_n}\).
\end{lemma}

\begin{proof}
Since \(\Mop_{\star,q}\) preserves each spherical harmonic degree and
\((\Id-\Pn)f\) has no constant component, we have
\(
  \ip{\Mop_{\star,q}(\Id-\Pn)f}{Y_{0,1}}=0.
\)
Thus, by \eqref{eq:star-split} and \(Y_{0,1}=(4\pi)^{-1/2}\),
\[
 \ip{u-\mathcal S_q(u_{n,\mathcal Q_{N_n}}^{\star})}{Y_{0,1}}
 =m_0^{\star,q}\ip{\Pn f-\Ln f}{Y_{0,1}} =\frac{m_0^{\star,q}}{\sqrt{4\pi}}
   \left(
     \int_{\Sph}f\,\dd\omega-\mathcal Q_{N_n}(f)
   \right).
\]
For every \(N_n\)-point quadrature rule, the Sobolev quadrature lower bound derived in
\cite{BrauchartHesse2007,HesseSloan2005Lower} provides
\(f_{N_n}\in H^r(\Sph)\) with
\(\|f_{N_n}\|_{H^r}=1\) such that
\[
  \left|
    \int_{\Sph}f_{N_n}\,\dd\omega-\mathcal Q_{N_n}(f_{N_n})
  \right|
  \geq c_rN_n^{-r/2}.
\]
Since \(Y_{0,1}=(4\pi)^{-1/2}\), for \(h\in C(\Sph)\),
\(
  \|h\|_{L^\infty}
  \geq \frac{1}{\sqrt{4\pi}}
  |\ip{h}{Y_{0,1}}|.
\)
Taking \(f=f_{N_n}\) therefore gives
\[
 E_{n,q}^{\star}(\mathcal Q_{N_n})
 \geq
 \|u-\mathcal S_q(u_{n,\mathcal Q_{N_n}}^{\star})\|_{L^\infty} \geq
 \frac{1}{\sqrt{4\pi}}
 \left|
 \ip{u-\mathcal S_q(u_{n,\mathcal Q_{N_n}}^{\star})}{Y_{0,1}}
 \right| \geq
 c_r|m_0^{\star,q}|N_n^{-r/2},
\]
after absorbing the factor \(4\pi\) into \(c_r\).
\end{proof}

\begin{remark}[When the constant mode is active]
Since \(P_0=1\), the Funk--Hecke formula gives
$\mu_0=2\pi\int_{-1}^{1}\Phi(t)\dd t$.
Thus \(\mu_0\ne0\) exactly when the zonal kernel has nonzero mean; in
particular, this holds for every nonzero kernel of one sign.  For the scaled
Laplace single-layer operator \(\Top=\kappa S\) discussed after
Assumption~\ref{ass:operator}, one has \(\mu_0=\kappa\).
By nonresonance, \(1-\mu_0\ne0\).  If \(\gamma_\star>1\), then
\(q+\varepsilon_\star>0\), and hence
\eqref{eq:nonzero-constant-mode} is equivalent to \(\mu_0\ne0\).
The case \((\star,q)=(G,0)\), for which
\(m_0^{G,0}=(1-\mu_0)^{-1}\ne0\), lies outside this regime since
\(\gamma_G=0\). The constant harmonic is convenient because its aliasing coefficient is the
ordinary scalar quadrature error, so the sharp quadrature lower bound applies
directly.  If \(\mu_0=0\), Lemma~\ref{lem:alias-lower} yields no information.
Saturation may still be detected through another fixed harmonic mode, but
this would require a corresponding weighted quadrature lower bound.
\end{remark}

\begin{theorem}[Sharp recovery--saturation dichotomy]
\label{thm:threshold-dichotomy}
Let Assumptions~\ref{ass:operator} and~\ref{ass:exact-order} hold, with $\mu_0\ne0$. Let
\(r>1\), \(q\in\Nzero\), and \(\star\in\{G,N\}\), and suppose that
\(
  \gamma_\star=(q+\varepsilon_\star)\beta>1.
\)
Let \(\{\mathcal Q_{N_n}\}\) be a family of positive \(t_n\)-exact
quadrature rules with
\begin{equation}\label{eq:quasi-optimal-design-regime}
  t_n\ge 2n,
  \qquad
  N_n\le C_1t_n^2.
\end{equation}
and set $s_n:=t_n-n$.
Denote by $E_{n,q}^{\star}:=E_{n,q}^{\star}(\mathcal Q_{N_n})$.
Then, for all sufficiently large \(n\),
\begin{equation}\label{eq:dichotomy-two-sided}
  E_{n,q}^{\star}
  \asymp
  n^{-(r+\gamma_\star-1)}+s_n^{-r}.
\end{equation}
Thus 
\begin{equation}\label{equ:critical}
    s_{n,\star}^{\rm crit} :=n^{(r+\gamma_\star-1)/r}.
\end{equation}
 is the sharp balance scale between recovery
of the spectral tail rate and aliasing saturation.
\end{theorem}

\begin{proof}
Let $a_n:=n^{-(r+\gamma_\star-1)}$ and
$b_n:=s_n^{-r}$.
Since \(\gamma_\star>1\), Theorem~\ref{thm:Q} gives
\begin{equation}\label{eq:dichotomy-upper}
  E_{n,q}^{\star}\le C(a_n+b_n).
\end{equation}
Under Assumption~\ref{ass:exact-order}, since $|\mu_\ell|\asymp(1+\ell)^{-\beta}$
and $\delta\le |1-\mu_\ell|\le1+C_{\Top}$,
$ |{\mu_\ell^{q+\varepsilon_\star}}/{1-\mu_\ell}|
  \asymp
  (1+\ell)^{-\gamma_\star}$
  as
  $\ell\to\infty$.
Hence Proposition~\ref{prop:sharp-tail} shows that
\begin{equation}\label{eq:dichotomy-tail-lower}
  \|\Mop_{\star,q}(\Id-\Pn)\|_{H^r\to L^\infty}
  \ge c_1a_n
\end{equation}
for all sufficiently large \(n\). We define
$A_n:=\Mop_{\star,q}(\Id-\Pn)$
and $C_n:=\Mop_{\star,q}(\Pn-\Ln)$
as operators from \(H^r(\Sph)\) to \(L^\infty(\Sph)\).  By
\eqref{eq:star-split},
\(
  E_{n,q}^{\star}
  =\|A_n+C_n\|_{H^r\to L^\infty}.
\)
The aliasing estimate used in Theorem~\ref{thm:Q}, together with
\(\gamma_\star>1\), gives
\begin{equation}\label{eq:dichotomy-alias-upper}
  \|C_n\|_{H^r\to L^\infty}\le C_2b_n.
\end{equation}
On the other hand, Lemma~\ref{lem:alias-lower} and
\(N_n\le C_1t_n^2\) yield
\(
  E_{n,q}^{\star}
  \ge c_r|m_0^{\star,q}|N_n^{-r/2}
  \gtrsim t_n^{-r}.
\)
Since \(t_n\ge2n\) and \(s_n=t_n-n\), we have
\(
{t_n}/{2}\le s_n\le t_n,
\)
and therefore \(t_n^{-r}\asymp s_n^{-r}=b_n\).  Thus
\begin{equation}\label{eq:dichotomy-alias-lower}
  E_{n,q}^{\star}\ge c_3b_n.
\end{equation}
It remains to combine the two lower bounds.  If
\(
  b_n\le {c_1a_n}/{(2C_2)}
\)
then the reverse triangle inequality and
\eqref{eq:dichotomy-tail-lower}--\eqref{eq:dichotomy-alias-upper} give
\(
  E_{n,q}^{\star}
  \ge \|A_n\|-\|C_n\|
  \ge c_1a_n-C_2b_n
  \ge c_1a_n/2.
\)
In this case \(a_n+b_n\lesssim a_n\), so
\(E_{n,q}^{\star}\gtrsim a_n+b_n\). Otherwise, if
\(
  a_n<{2C_2b_n}/{c_1},
\)
we have \(a_n+b_n\lesssim b_n\).  By
\eqref{eq:dichotomy-alias-lower},
\(
  E_{n,q}^{\star}\ge c_3b_n
  \gtrsim a_n+b_n.
\)
Thus in either case
\[
  E_{n,q}^{\star}\gtrsim a_n+b_n.
\]
Together with \eqref{eq:dichotomy-upper}, this proves
\eqref{eq:dichotomy-two-sided}.
Finally, $a_n=b_n$ gives $ s_n=n^{(r+\gamma_\star-1)/r}$, the  threshold \eqref{equ:critical}.
\end{proof}

\begin{remark}[The two branches under weaker assumptions]
\label{rem:weaker-dichotomy}
The two branches of Theorem~\ref{thm:threshold-dichotomy} require weaker
assumptions when considered separately.  Under
Assumption~\ref{ass:operator} alone, if
\(s_n\gtrsim s_{n,\star}^{\rm crit}\), then
Theorem~\ref{thm:Q} gives
\(
  E_{n,q}^{\star}
  \lesssim
  n^{-(r+\gamma_\star-1)}.
\)
Likewise, if \(\mu_0\ne0\) and \(N_n\lesssim t_n^2\), then
\(E_{n,q}^{\star}\asymp s_n^{-r}\) whenever
\(s_n\lesssim s_{n,\star}^{\rm crit}\), without Assumption~\ref{ass:exact-order}.  The exact-order condition is
needed to make the spectral-tail branch sharp and hence to obtain the global
two-sided estimate \eqref{eq:dichotomy-two-sided}.
\end{remark}

\begin{remark}[Relation to the literature]
Quadrature conditions for preserving superconvergence are classical.  In the
spline setting of Atkinson and Bogomolny, for example, the once-iterated
discrete Galerkin solution satisfies a bound of
the form
\(
 \|x-z_h^*\|\leq C\bigl(h^{d+1}+h^{2r}\bigr),
\)
where $d$ is the degree of precision of the local quadrature; recovering the
full one-step rate $h^{2r}$ therefore requires $d\geq2r-1$
\cite{AtkinsonBogomolny1987}; see also \cite{GolbergBowman1998}.  Re-iterated
and repeated methods are studied in
\cite{LongNelakanti2007,PorterStirling1993}.

For the spherical multiplier problem, the critical scale
\(
 s_{n,\star}^{\rm crit}
 =n^{1+\{(q+\varepsilon_\star)\beta-1\}/r}
\)
displays the dependence on the Sloan depth and the exact smoothing order.
Proposition~\ref{prop:sharp-tail} shows that the ideal tail rate has the
stated order, while the constant-mode argument turns the Sobolev quadrature
lower bound into the matching obstruction.  Thus the sufficient scale is
sharp under the exact-order, nondegeneracy, and cardinality assumptions. 

\end{remark}

\begin{remark}[Admissible exact rules]
Quadrature rules satisfying \eqref{eq:quasi-optimal-design-regime} exist.
Indeed, equal-weight spherical \(t_n\)-designs can be chosen with
\(N_n\asymp t_n^2\): the classical lower bound is due to \cite{delsarte1991geometriae}, while
the matching existence bound can be found in \cite{bondarenko2013optimal}.
\end{remark}

\section{Sloan iteration with inexact quadrature}
\label{sec:inexact}
Section~\ref{sec:discrete} used polynomial exactness to construct $\mathcal{L}_n$ onto \(\mathbb P_n\).  We now consider a family
of positive quadrature rules \(\{\mathcal Q_{N_n}\}_{n\geq1}\) without
assuming polynomial exactness.  For fixed \(n\), we denote by
\(\mathcal Q=\mathcal Q_{N_n}\) for simplicity.  We assume the Marcinkiewicz--Zygmund
condition
\begin{equation}\label{eq:MZ-condition}
 (1-\eta_n)\|p\|_{L^2}^2
 \leq \mathcal Q_{N_n}(|p|^2)
 \leq (1+\eta_n)\|p\|_{L^2}^2,
 \qquad p\in\mathbb P_n,\qquad 0\leq\eta_n<1.
\end{equation}
Such inequalities provide the discrete norm equivalence underlying hyperinterpolation, stable
spherical least squares and quadrature-based Galerkin methods; see, e.g.,
\cite{an2022quadrature,an2024bypassing,MR4780349,filbir2011marcinkiewicz,MR4124840,LuWang2023,
Mhaskar2006Weighted,mhaskar2001spherical,Wu2026MZGalerkin}. Let \(\{Y_a\}_{a=1}^{(n+1)^2}\) be an orthonormal basis of \(\mathbb P_n\),
and define the Gram matrix $G_n\in\mathbb{R}^{(n+1)^2\times (n+1)^2}$ by
$(G_n)_{ab}:=\mathcal Q_{N_n}(Y_aY_b)$.
For \(p=\sum_a c_aY_a\), \eqref{eq:MZ-condition} is equivalent to
\[
  (1-\eta_n)\|c\|_{\ell^2}^2
  \leq c^{\mathsf T}G_nc
  \leq(1+\eta_n)\|c\|_{\ell^2}^2.
\]
Equivalently, the spectrum $\sigma(G_n)$ of $G_n$ satisfies $\sigma(G_n)\subset[1-\eta_n,1+\eta_n]$, and the condition number of $G_n$ satisfies $ \kappa_2(G_n)\leq(1+\eta_n)/(1-\eta_n)$. In particular, $\eta_n<1$ makes $G_n$ positive definite.

When \(G_n\neq I\), the hyperinterpolation operator
\eqref{eq:hyperinterpolation} need not reproduce \(\mathbb P_n\).  We replace
it by the weighted least-squares projector
\begin{equation*}
 \widetilde{\mathcal L}_{n,\mathcal Q}v
 :=\underset{p\in\mathbb P_n}{\operatorname{argmin}}
   \sum_{j=1}^{N_n}w_{j,n}|v(x_{j,n})-p(x_{j,n})|^2.
\end{equation*}
The lower bound in \eqref{eq:MZ-condition} ensures uniqueness, and the normal
equations are
\begin{equation}\label{eq:MZ-projection}
 \Qip{\widetilde{\mathcal L}_{n,\mathcal Q}v}{p}
 =\Qip{v}{p},
 \qquad p\in\mathbb P_n.
\end{equation}
Since the least-squares residual vanishes for \(v\in\mathbb P_n\),
\(\widetilde{\mathcal L}_{n,\mathcal Q}\) reproduces \(\mathbb P_n\), and
therefore
\(
  \widetilde{\mathcal L}_{n,\mathcal Q}^2
  =\widetilde{\mathcal L}_{n,\mathcal Q}.
\)
If \(\mathcal Q_{N_n}\) is exact on \(\mathbb P_{2n}\), then \(G_n=I\) and
\(\widetilde{\mathcal L}_{n,\mathcal Q}\) reduces to the usual
hyperinterpolant \(\Ln\).

\begin{lemma}[Weighted least-squares approximation]
\label{lem:MZ-least-squares}
Suppose that \eqref{eq:MZ-condition} holds.  If $v\in H^r(\Sph)$ with
$r>1$, then
\begin{equation*}
 \|\Pn v-\widetilde{\mathcal L}_{n,\mathcal Q}v\|_{L^2}
 \leq \|v-\widetilde{\mathcal L}_{n,\mathcal Q}v\|_{L^2} \leq C
 \left[1+\left(\frac{1+\eta_n}{1-\eta_n}\right)^2\right]^{1/2}
 n^{-r}\|v\|_{H^r}.
\end{equation*}
The constant $C$ is independent of $n$, $v$, and $\eta_n$.
\end{lemma}

\begin{proof}
Since $v-\Pn v$ is orthogonal to $\mathbb P_n$ and
$\Pn v-\widetilde{\mathcal L}_{n,\mathcal Q}v\in\mathbb P_n$,
\[
 \|v-\widetilde{\mathcal L}_{n,\mathcal Q}v\|_{L^2}^2
 =\|v-\Pn v\|_{L^2}^2
  +\|\Pn v-\widetilde{\mathcal L}_{n,\mathcal Q}v\|_{L^2}^2.
\]
This proves the first inequality.  The second is the optimal weighted
least-squares estimate on the sphere
\cite[Theorem~1.2]{LuWang2023}, together with
$\kappa_2(G_n)\leq(1+\eta_n)/(1-\eta_n)$.
\end{proof}

Fix $r>1$ and $f\in H^r(\Sph)$. We define the Gram-corrected discrete Galerkin solution
$\widetilde u_{n,\mathcal Q}^G\in\mathbb P_n$ and the corresponding
product-integration Nystr\"om solution
$\widetilde u_{n,\mathcal Q}^N\in H^r(\Sph)$ by
\begin{equation}\label{eq:MZ-Galerkin}
 \Qip{(\Id-\Top)\widetilde u_{n,\mathcal Q}^G}{p}
 =\Qip{f}{p}\quad \forall p\in\mathbb P_n,
 \end{equation}
and
\begin{equation}\label{eq:MZ-Nystrom}
 (\Id-\Top\widetilde{\mathcal L}_{n,\mathcal Q})
 \widetilde u_{n,\mathcal Q}^N
 =f,
\end{equation}
respectively.

\begin{proposition}[Gram-corrected solution]
\label{prop:MZ-representations}
Suppose that Assumption~\ref{ass:operator} holds and the quadrature family
satisfies \eqref{eq:MZ-condition}.  Let \(r>1\).  For every
\(f\in H^r(\Sph)\), problem~\eqref{eq:MZ-Galerkin} has a unique solution in
\(\mathbb P_n\).  Moreover,
\(\Id-\Top\widetilde{\mathcal L}_{n,\mathcal Q}\) is boundedly invertible on
\(H^r(\Sph)\), so \eqref{eq:MZ-Nystrom} also has a unique solution.  The two
solutions satisfy
\begin{equation}\label{eq:MZ-closed-forms}
 \widetilde u_{n,\mathcal Q}^G
 =\Rop\widetilde{\mathcal L}_{n,\mathcal Q}f,
 \qquad
 \widetilde u_{n,\mathcal Q}^N
 =f+\Bop\widetilde{\mathcal L}_{n,\mathcal Q}f.
\end{equation}
Consequently,
\begin{equation}\label{eq:MZ-direct-errors}
 u-\widetilde u_{n,\mathcal Q}^G
 =\Rop(\Id-\widetilde{\mathcal L}_{n,\mathcal Q})f,
 \qquad
 u-\widetilde u_{n,\mathcal Q}^N
 =\Bop(\Id-\widetilde{\mathcal L}_{n,\mathcal Q})f.
\end{equation}
\end{proposition}

\begin{proof}
Denote by \(
\widetilde{\mathcal L}
=\widetilde{\mathcal L}_{n,\mathcal Q}.
\)
For the discrete Galerkin problem, let
\(
w=(\Id-\Top)\widetilde u_{n,\mathcal Q}^G.
\)
Since \(\Top\mathbb P_n\subseteq\mathbb P_n\), both \(w\) and
\(\widetilde{\mathcal L}f\) belong to \(\mathbb P_n\).  By
\eqref{eq:MZ-Galerkin} and the normal equations
\eqref{eq:MZ-projection}, we have $\Qip{w-\widetilde{\mathcal L}f}{p}=0$
for all $p\in\mathbb P_n$.
Taking \(p=w-\widetilde{\mathcal L}f\) and using the lower
Marcinkiewicz--Zygmund bound gives
$w=\widetilde{\mathcal L}f$.
Hence any solution must satisfy
\(
  \widetilde u_{n,\mathcal Q}^G
  =\Rop\widetilde{\mathcal L}f.
\)
Conversely, \(\Rop\widetilde{\mathcal L}f\in\mathbb P_n\) and satisfies
\eqref{eq:MZ-Galerkin}, proving existence and uniqueness.

For the product-integration Nystr\"om problem, let
\(\widetilde u=\widetilde u_{n,\mathcal Q}^N\).  Since
\(\widetilde{\mathcal L}\) reproduces \(\mathbb P_n\) and
\(\Top\widetilde{\mathcal L}\) takes values in \(\mathbb P_n\),
applying \(\widetilde{\mathcal L}\) to
\(
  (\Id-\Top\widetilde{\mathcal L})\widetilde u=f
\)
gives
\(
  (\Id-\Top)\widetilde{\mathcal L}\widetilde u
  =\widetilde{\mathcal L}f.
\)
Thus
\(
  \widetilde{\mathcal L}\widetilde u
  =\Rop\widetilde{\mathcal L}f,
\)
and substitution into the original equation yields
\(
  \widetilde u
  =f+\Top\Rop\widetilde{\mathcal L}f
  =f+\Bop\widetilde{\mathcal L}f.
\)
Conversely, this function satisfies the Nystr\"om equation, so the solution is
unique and
\(
  (\Id-\Top\widetilde{\mathcal L})^{-1}
  =\Id+\Bop\widetilde{\mathcal L}.
\)
Since \(r>1\) implies \(H^r(\Sph)\hookrightarrow C(\Sph)\),
\(\widetilde{\mathcal L}:H^r(\Sph)\to\mathbb P_n\) is bounded, and hence so
is this inverse on \(H^r(\Sph)\).

Finally, subtracting the two representations from
\(u=\Rop f=f+\Bop f\) gives \eqref{eq:MZ-direct-errors}.
\end{proof}

\begin{theorem}[Sloan estimates with inexact quadrature]
\label{thm:MZ}
Under Assumption~\ref{ass:operator}, suppose the quadrature family satisfies
\eqref{eq:MZ-condition}.  Fix $q\in\Nzero$ and $\star\in\{G,N\}$.  For
$f\in H^r(\Sph)$ with $r>1$, set
$\gamma_\star=(q+\varepsilon_\star)\beta$.  Then
\begin{equation}\label{eq:MZ-Sloan-rate}
 \|u-\mathcal S_q(
   \widetilde u_{n,\mathcal Q}^{\star})\|_{L^\infty}
 \leq C\left\{n^{-(r+\gamma_\star-1)}
   +\rho_{\gamma_\star}(n)
    \left[1+\left(\frac{1+\eta_n}{1-\eta_n}\right)^2\right]^{1/2}
    n^{-r}\right\}\|f\|_{H^r}.
\end{equation}
The constant $C$ is independent of $n$, $f$, and $\eta_n$.
\end{theorem}

\begin{proof}
Applying the Sloan error identity to \eqref{eq:MZ-direct-errors} gives
$u-\mathcal S_q(\widetilde u_{n,\mathcal Q}^{\star})
 =\Mop_{\star,q}(\Id-\Pn)f
  +\Mop_{\star,q}
   (\Pn f-\widetilde{\mathcal L}_{n,\mathcal Q}f)$.
Lemma~\ref{lem:tail} bounds the first term in the error decomposition by
$Cn^{-(r+\gamma_\star-1)}\|f\|_{H^r}$, while
Lemmas~\ref{lem:MZ-least-squares} and \ref{lem:low} bound the second by
\[
 C\rho_{\gamma_\star}(n)
 \left[1+\left(\frac{1+\eta_n}{1-\eta_n}\right)^2\right]^{1/2}
 n^{-r}\|f\|_{H^r}.
\]
Their sum is \eqref{eq:MZ-Sloan-rate}.
\end{proof}
\begin{remark}[Marcinkiewicz--Zygmund stability and saturation]
The saturation conclusion of Theorem~\ref{thm:threshold-dichotomy} extends
to Gram-corrected inexact quadrature. We define
\(
  \widetilde{\mathcal Q}_{n,N_n}(f)
  :=\sqrt{4\pi}\,
    \ip{\widetilde{\mathcal L}_{n,\mathcal Q}f}{Y_{0,1}}.
\)
This is a linear quadrature rule based on the \(N_n\) sampled values of \(f\),
and
\[
  \ip{\Pn f-\widetilde{\mathcal L}_{n,\mathcal Q}f}{Y_{0,1}}
  =
  \frac{1}{\sqrt{4\pi}}
  \left(
    \int_{\Sph}f\dd\omega
    -\widetilde{\mathcal Q}_{n,N_n}(f)
  \right).
\]
Hence the general \(N_n\)-point Sobolev quadrature lower bound
\cite{HesseSloan2005Lower}, together with the constant-mode argument of
Lemma~\ref{lem:alias-lower}, gives
\[
  \sup_{0\ne f\in H^r(\Sph)}
  \frac{
    \|u-\mathcal S_q(\widetilde u_{n,\mathcal Q}^{\star})\|_{L^\infty}
  }{
    \|f\|_{H^r}
  }
  \geq
  c_r|m_0^{\star,q}|N_n^{-r/2}.
\]
Therefore, if \(\gamma_\star>1\), \(\mu_0\ne0\),
\(\sup_n\eta_n<1\), and \(N_n\leq Cn^2\), Theorem~\ref{thm:MZ} yields the
sharp worst-case rate \(n^{-r}\).  Thus Marcinkiewicz--Zygmund-stable sampling of quasi-optimal
cardinality preserves the saturation branch, but Marcinkiewicz--Zygmund stability alone does not
imply recovery of the quadrature-free rate; that requires additional control
of \(\Pn f-\widetilde{\mathcal L}_{n,\mathcal Q}f\).
\end{remark}

\section{Numerical illustration}\label{sec:numerics}

The numerical experiments illustrate the two main conclusions of the analysis.
The first examines the competition between the filtered spectral tail and
quadrature-induced aliasing under polynomially exact quadrature, focusing on
the recovery--saturation transition in
Theorem~\ref{thm:threshold-dichotomy}.  The second considers
Marcinkiewicz--Zygmund stable point sets without prescribed polynomial
exactness and compares their Sloan errors with those obtained from spherical
designs.

For all computations, we index the real spherical harmonics of degree \(\ell\)
by \(m=-\ell,\ldots,\ell\), with \(m=0\) denoting the zonal harmonic.  We set
\(r=3\) and use the multiplier operator $\Top$ satisfying
$\widehat{\Top v}_{\ell,m}
  =\mu_\ell\widehat v_{\ell,m}$
  and
$\mu_\ell=0.55(1+\ell)^{-3/2}$.
This defines a Hilbert--Schmidt zonal integral operator with exact smoothing
order \(\beta=3/2\), and
\(
  |1-\mu_\ell|\ge0.45.
\)
Thus \(\Top\) satisfies Assumptions~\ref{ass:operator}
and~\ref{ass:exact-order}.  Each reported rate \(p\) is the negative slope of
a least-squares fit of \(\log e_n\) against \(\log n\) over the five largest
values of \(n\).

\subsection{Exact quadrature: recovery and saturation}

We restrict this experiment to the product-integration Nystr\"om method.
By the one-layer relation in Remark~\ref{rem:onelayer}, its Sloan error at
depth \(q\) coincides with the discrete Galerkin Sloan error at depth
\(q+1\).  A separate Galerkin experiment would therefore reproduce the same
curves with a shifted Sloan depth.
For the product-integration Nystr\"om method, set
$\gamma_q=(q+1)\beta$
and
$s=t-n$.
The two terms in Theorem~\ref{thm:Q} then have orders
$n^{-(r+\gamma_q-1)}$
and
$s^{-r}$.
For \(q=0\), the spectral-tail exponent is \(7/2\) and the recovery scale is
\(s\asymp n^{7/6}\).  After one Sloan step, these become \(5\) and
\(s\asymp n^{5/3}\), respectively.  For a datum \(f\),
\[
 u-\mathcal S_q(u_{n,\mathcal Q}^{N})
 =
 \underbrace{\Mop_{N,q}(\Id-\Pn)f}_{\text{filtered spectral tail}}
 +
 \underbrace{\Mop_{N,q}(\Pn f-\Ln f)}_{\text{filtered aliasing}}.
\]

We compare two quantities associated with this decomposition.  The exact
\(H^3(\Sph)\)-to-\(L^\infty(\Sph)\) norm of the filtered spectral tail is
\begin{equation}
 \tau_n^{(q)}
 :=
 \sup_{\|f\|_{H^3}\leq1}
 \|\Mop_{N,q}(\Id-\Pn)f\|_{L^\infty}
 =
 \left\{
 \sum_{\ell>n}
 \bigl(1+\ell(\ell+1)\bigr)^{-3}
 \left|\frac{\mu_\ell^{q+1}}{1-\mu_\ell}\right|^2
 \frac{2\ell+1}{4\pi}
 \right\}^{1/2}.
\end{equation}
The identity follows from \eqref{eq:sobolev} and the addition formula
\eqref{eq:addition}.  Numerically, the series is truncated at degree
\(50000\).
To measure aliasing, we retain the constant output mode and restrict the
input to the band \(t<\ell\leq L\).  Exactness through degree \(t\) removes
the constant-mode quadrature defect generated by all input harmonics of
degree at most \(t\).  We therefore define
\begin{equation}\label{equ:lowerbound}
\begin{aligned}
 \alpha_{n,t}^{(q,L)}
 &:=
 \sup_{\substack{\|f\|_{H^3}\leq1\\
       \widehat f_{\ell,m}=0\ \text{unless }t<\ell\leq L}}
 \left\|
  \ip{\Mop_{N,q}(\Pn f-\Ln f)}{Y_{0,1}}Y_{0,1}
 \right\|_{L^\infty}
 \\
 &=
 \frac1{4\pi}
 \left|\frac{\mu_0^{q+1}}{1-\mu_0}\right|
 \left\{
 \sum_{\ell=t+1}^{L}
 \bigl(1+\ell(\ell+1)\bigr)^{-3}
 \sum_{m=-\ell}^{\ell}
   \abs{\mathcal Q_N(Y_{\ell,m})}^2
 \right\}^{1/2}.
\end{aligned}
\end{equation}
Since the filtered spectral tail has no constant component and the
\(L^\infty\)-norm dominates the norm of the constant projection,
$ E_{n,q}^{N}(\mathcal Q_N)
 \geq \alpha_{n,t}^{(q,L)}$.
 Thus \(\alpha_{n,t}^{(q,L)}\) is a computable lower-bound certificate for the
full worst-case Sloan error, using only one output mode and a finite input
band.  The Sloan depth \(q\) enters only through the fixed factor
\(
  |{\mu_0^{q+1}}/{(1-\mu_0)}|,
\)
and therefore changes the magnitude of the certificate but not its
asymptotic rate.
For each prescribed degree \(t\), we use a well-conditioned spherical
\(t\)-design \cite{MR2763659} with \(N=(t+1)^2\) equal-weight nodes, so that
the associated rule is exact on \(\mathbb P_t\).  Figure~\ref{fig:num-dichotomy}
shows three representative cases.

\begin{figure}[htbp]
  \centering
  \includegraphics[width=\textwidth]
    {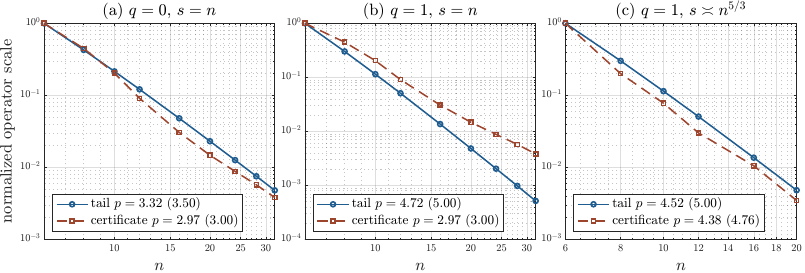}
  \caption{Filtered spectral tail and constant-mode aliasing certificate under exact
quadrature.  Parentheses indicate the theoretical tail exponent for solid
curves and the fitted \(s^{-3}\) slope along the displayed \((n,t)\)-schedule
for dashed curves.}
  \label{fig:num-dichotomy}
\end{figure}

In the first two panels, \(t=2n\), hence \(s=n\), and we take
\(q=0\) and \(q=1\), respectively, with
\(
 n=6,8,10,12,16,20,24,28,32.
\)
In the third panel, we test the \(q=1\) recovery scale
\(s\asymp n^{5/3}\).  Using the available design degrees, we take
\(
 (n,t)=(6,25),(8,40),(10,56),(12,77),(16,110),(20,160),
\)
for which
\(
 s=t-n=19,32,46,65,94,140.
\)
For $L$ in the constant-mode certificate \eqref{equ:lowerbound}, we set
\(
  L=t+\max\{20,\lceil t/2\rceil\},
\)
so that the input band contains at least \(20\) degrees and has width
proportional to \(t\) for large \(t\).  Each curve is normalized by its
value at \(n=6\); this changes only its vertical position and not the fitted
decay rate.

At minimal exactness, one Sloan step increases the observed tail rate from
\(3.32\) to \(4.72\), while the certificate rate remains \(2.97\), close to
the sharp aliasing rate \(n^{-3}\).  Theorem~\ref{thm:threshold-dichotomy}
gives the matching worst-case order \(n^{-3}\) for this quasi-optimal design
family, and the computation shows that the obstruction is already visible
in the constant mode.  At the recovery scale \(s\asymp n^{5/3}\), the
observed tail and certificate rates are \(4.52\) and \(4.38\), respectively,
consistent with the recovery predicted by the theorem.

\subsection{Inexact quadrature: Gram-corrected Sloan errors}

We next consider the axisymmetric datum \(f\) defined by $ \widehat f_{\ell,m}=0$ for $m\ne0$, $\widehat f_{\ell,0}=(1+\ell)^{-3.65}$ for $\ell\ge1$,
and with \(\widehat f_{0,0}=0.8\).  This datum belongs to \(H^3(\Sph)\) and
lies close to the regularity threshold: if
\(\widehat f_{\ell,0}\asymp\ell^{-a}\), then membership in \(H^3(\Sph)\)
requires \(a>7/2\), whereas here \(a=3.65\).

For \(n=4,6,8,10,12,16,20,24\), we compare four equal-weight point
families: a spherical \(2n\)-design, minimal-energy points
\cite{RakhmanovSaffZhou1994}, maximal-determinant points
\cite{WomersleySloan2001Interpolation}, and recursive zonal equal-area
points \cite{Leopardi2006EqualArea}.  Each set contains
\(N=(2n+1)^2\) points with weights \(4\pi/N\).

For each point set, we form the Gram-corrected least-squares approximation
from Section~\ref{sec:inexact} and compute the Nystr\"om Sloan error for
\(q=0,1,2\).  Specifically, we evaluate
$u-\mathcal S_q(\widetilde u_{n,\mathcal Q}^{N})
  =\Mop_{N,q}(\Id-\widetilde{\mathcal L}_{n,\mathcal Q})f$
and approximate its \(L^\infty\)-norm by the maximum absolute value over
\(6000\) Fibonacci points together with the two poles.  The spherical
harmonic expansion is truncated at degree \(320\) at both the quadrature
and test points. 

Table~\ref{tab:num-inexact} reports
\(\eta_n=\|G_n-I\|_2\), the condition number \(\kappa_2(G_n)\), and the
observed convergence rates.  The first two columns give the maximum values
of \(\eta_n\) and \(\kappa_2(G_n)\), respectively, over all tested \(n\).
The corresponding error curves are shown in
Figure~\ref{fig:num-inexact-sloan-errors}.

\begin{table}[H]
\caption{Marcinkiewicz--Zygmund diagnostics and observed Sloan error rates.}
\label{tab:num-inexact}
\centering
\begin{tabular}{lccccc}
\toprule
point family & \(\max\eta_n\) & \(\max\kappa_2(G_n)\)
 & \(q=0\) & \(q=1\) & \(q=2\)\\
\midrule
spherical \(2n\)-design & \(7.3\times10^{-12}\) & \(1.000\)
 & \(3.34\) & \(3.84\) & \(3.84\)\\
minimal energy & \(0.0292\) & \(1.058\)
 & \(3.26\) & \(4.03\) & \(4.05\)\\
maximal determinant & \(0.0669\) & \(1.142\)
 & \(3.30\) & \(4.37\) & \(4.46\)\\
equal area & \(0.0719\) & \(1.100\)
 & \(3.28\) & \(3.99\) & \(4.03\)\\
\midrule
spectral tail & --- & ---
 & \(3.31\) & \(4.67\) & \(6.03\)\\
\bottomrule
\end{tabular}
\end{table}

\begin{figure}[H]
  \centering
  \includegraphics[width=\textwidth]
    {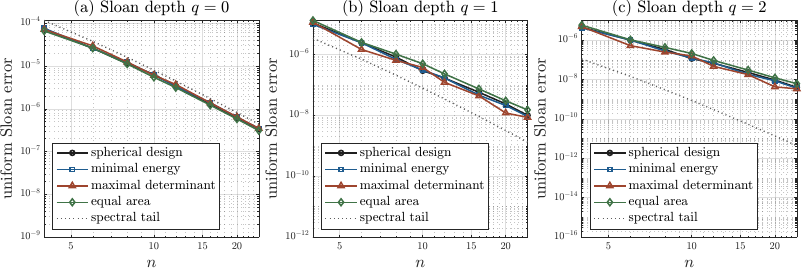}
  \caption{Sampled uniform errors for the Gram-corrected
  product-integration Nystr\"om method.}
  \label{fig:num-inexact-sloan-errors}
\end{figure}

For all three inexact families, \(\eta_n<0.072\) and
\(\kappa_2(G_n)<1.142\), while the resulting errors remain comparable to
those obtained with the spherical design.  Increasing the Sloan depth from
\(q=0\) to \(q=1\) improves the observed convergence rate for every point
family.  From \(q=1\) to \(q=2\), however, the filtered-tail rate rises from
\(4.67\) to \(6.03\), whereas the total-error rates change only slightly,
from \(3.84\)--\(4.37\) to \(3.84\)--\(4.46\).  The total error is therefore
dominated by the low-degree sampling defect rather than the spectral tail,
in agreement with the saturation mechanism predicted by the
Marcinkiewicz--Zygmund analysis in Section~\ref{sec:inexact}.

\section{Concluding remarks}

Sloan iteration and its saturation under quadrature have a common spectral
explanation.  Without quadrature, the error is confined to unresolved
frequencies and successive Sloan steps exploit the smoothing of the integral
operator.  Quadrature introduces a resolved-mode aliasing component, on which
further multiplier decay need not improve the asymptotic rate.  For
polynomially exact quadrature we quantified this competition by the sharp
two-sided estimate
\[
  E_{n,q}^{\star}
  \asymp
  n^{-(r+\gamma_\star-1)}+s_n^{-r},
\]
which yields the depth-dependent recovery--saturation threshold.  Under
Marcinkiewicz--Zygmund sampling, Gram correction preserves stability and the
sharp saturation branch, although recovery requires additional control of
the aliasing term.

The tail--aliasing mechanism is not specific to spherical harmonics.  For a
\(d\)-dimensional spectral system with the corresponding local
spectral-counting estimate, the tail exponent becomes
\(r+\gamma-d/2\).  This suggests analogous results for Laplace spectral
multipliers on compact manifolds, Fourier and orthogonal-polynomial
discretizations, with additional mode-coupling estimates for approximately
diagonal operators.  Increasing the Sloan depth therefore requires a
corresponding increase in discretization resolution if the additional
smoothing is to remain visible.

\small
\bibliographystyle{siamplain}
\bibliography{myref}

\end{document}